\documentclass[reqno,12pt]{amsart}
\usepackage{amscd,amsfonts,mathrsfs,amsthm,amssymb,amsmath,mathtools}
\usepackage{upgreek}
\usepackage{csquotes}
\usepackage{xcolor}
\usepackage{enumerate}
\usepackage{bbm}
\usepackage{multicol}
\usepackage{stmaryrd,color}
\usepackage{array}
\usepackage{ae}
\usepackage{accents}
\usepackage{epsfig}
\usepackage[e]{esvect}
\usepackage[nobysame]{amsrefs}
\usepackage[margin=1in]{geometry}
\usepackage[toc,title,page]{appendix}
\usepackage{hyperref}
\hypersetup{colorlinks=true,linkcolor=blue,filecolor=blue,urlcolor=blue, citecolor=blue}

\let\oldtocsection=\tocsection
\let\oldtocsubsection=\tocsubsection
\let\oldtocsubsubsection=\tocsubsubsection
\renewcommand{\tocsection}[2]{\hspace{0em}\oldtocsection{#1}{#2}}
\renewcommand{\tocsubsection}[2]{\hspace{1em}\oldtocsubsection{#1}{#2}}
\renewcommand{\tocsubsubsection}[2]{\hspace{2em}\oldtocsubsubsection{#1}{#2}}

\def\equationcolor {\color{black}}
\def\textcolor     {\color{black}}

\def\bcoleq    {\begin{equation}\equationcolor}
\def\ecoleq    {\textcolor\end{equation}}
\def\bcoleqn   {\equationcolor\begin{eqnarray}}
\def\ecoleqn   {\end{eqnarray}\textcolor}

\def\S        {\mathbb{S}}

\def\Re{\operatorname{Re}}
\def\Im{\operatorname{Im}}

\def\span{\operatorname{span}}

\def\R{\mathbb{R}}
\def\N{\mathbb{N}}

\newcommand{\disp}{\displaystyle}

\let\epsilon\varepsilon
\let\rho\varrho
\let\tends\longrightarrow

\DeclareMathOperator*{\rank}{rank}

\newtheorem{theorem}{Theorem}[section]
\newtheorem{mythm}{Theorem}

\newtheorem{lemma}[theorem]{Lemma}

\theoremstyle{definition}
\newtheorem*{assumption*}{$\lambda_{1}$-Condition}

\newtheorem*{conj}{Conjecture}
\newtheorem{remark}[theorem]{Remark}

\def\pproof#1{\@ifnextchar[\opargproof
{\opargproof[\it Proof of #1.]}}
\def\opargproof[#1]{\par\noindent {\bf #1 }}

\makeatother

\numberwithin{equation}{section}

\begin{document}

\title[Harmonic maps between spheres]{Harmonic maps from $\S^3$ to $\S^2$
and the rigidity of the Hopf fibration}
\author[A. Georgakopoulos]{\textsc{A. Georgakopoulos}}
\author[M. Magliaro]{\textsc{M. Magliaro}}
\author[L. Mari ]{\textsc{L. Mari}}
\author[A. Savas-Halilaj]{\textsc{A. Savas-Halilaj}}

\address{Athanasios Georgakopoulos\newline
{Department of Mathematics,
Section of Algebra \!\&\! Geometry, \!\!
University of Ioannina,
45110 Ioannina, Greece} \newline
{\sl E-mail:} {\bf thanosgeo18@gmail.com}
}

\address{Marco Magliaro \newline
Dipartimento di Scienza e Alta Tecnologia,
Universit\`a degli Studi dell' Insubria,
22100 Como, Italy\newline
{\sl e-mail:} {\bf marco.magliaro@uninsubria.it}
}

\address{Luciano Mari \newline
Dipartimento di Matematica ``Federigo Enriques",
Universit\`a degli Studi di Milano,
20133 Milano, Italy\newline
{\sl e-mail:} {\bf luciano.mari@unimi.it}
}

\address{Andreas Savas-Halilaj\newline
{Department of Mathematics,
Section of Algebra \!\&\! Geometry, \!\!
University of Ioannina,
45110 Ioannina, Greece} \newline
{\sl E-mail:} {\bf ansavas@uoi.gr}
}

\renewcommand{\subjclassname}{  \textup{2000} Mathematics Subject Classification}
\thanks{A. Georgakopoulos \& A. Savas-Halilaj are supported in the framework of H.F.R.I. call ``Basic Research Financing” under the
National Recovery and Resilience Plan “Greece 2.0" funded by the European Union-NextGenerationEU
(HFRI Project Number: 14758). L. Mari is supported by the PRIN 20225J97H5 ``Differential-geometric aspects of manifolds via Global Analysis"}
\parindent = 0 mm
\hfuzz     = 6 pt
\parskip   = 3 mm
\date{}

\begin{abstract}
It was conjectured by Eells that the only harmonic maps $f : \S^3 \to \S^2$ are Hopf fibrations composed with conformal maps of $\S^2$. We support this conjecture by proving its validity under suitable conditions on the Hessian and the singular values of $f$. Among the  results, we obtain a pinching theorem in the spirit of that of Simons, Lawson and Chern, do Carmo and Kobayashi for minimal hypersurfaces in the sphere.
\end{abstract}

\maketitle
\setcounter{tocdepth}{1}
\section{Introduction}
A smooth map $f : M \to N$ between compact Riemannian manifolds is called \emph{harmonic} if it is a critical point of the energy functional
\[
\mathcal{E}(f) \doteq \frac{1}{2} \int_M |df|^2 \, d\mu,
\]
where $|df|$ denotes the norm of the differential $df$ with respect to the Riemannian metrics
on $M$ and $N$, and $d\mu$ is the volume form of $M$.
A classical question in the theory of harmonic maps is whether every continuous map admits a harmonic
representative in its homotopy class. Eells and Sampson \cite{eells} introduced the heat flow to deform a given map into a harmonic one, which leads to a harmonic representative if the target manifold is non-positively curved, but when $N$ has positive curvature the existence problem is more subtle. In particular, the classification of harmonic maps between spheres is far from being completely understood.

In 1931 Hopf \cite{hopf} constructed a remarkable example of a homotopically nontrivial harmonic map
$\pi : \S^3 \to \S^2$ now known as the Hopf fibration, see Section~\ref{sec_hopf}.
According to results of Eells, Ferreira and Ratto \cite{ratto2,eells1,eells2}, for every homotopy class in $[\S^3,\S^2]$
there exist Riemannian metrics on $\S^3$ and $\S^2$ for which this class
admits a harmonic representative.
In contrast with the previous
results,
for the standard round metrics on $\S^3$ and $\S^2$ the only known examples of harmonic maps
$f : \S^3 \to \S^2$ are compositions of the Hopf fibration with conformal maps of $\S^2$. Indeed, Eells formulated the following conjecture \cite[Note~10.4.1, p.~422]{baird1} and \cite[p.~730]{wangg}:

\begin{conj}[Eells]
{\em Every harmonic map $f : \S^3 \to \S^2$ (with their round metrics) factors through the Hopf fibration as
$
f = g\circ \pi,
$
where $g : \S^2 \to \S^2$ is a conformal map.}
\end{conj}
Currently, the conjecture is still widely open. The purpose of the present paper is to support its validity by proving it under some further mild assumptions on $f$. We shall investigate harmonic maps $f : \S^3 \to \S^2$ satisfying suitable bounds on the Hessian $B = \nabla d f$ and the singular values
$\lambda_1\ge \lambda_2\ge\lambda_3$
of $df$, i.e. the eigenvalues of $\sqrt{df^* df}$. Note that, for dimensional reasons, we have that $\lambda_3 = 0$.
In particular, the \emph{$2$-dilation} of $f$ is defined by
\[
\operatorname{D}_2 \doteq \lambda_1 \lambda_2,
\]
The quantity $\operatorname{D}_2$ measures the area distortion of $f$ on tangent $2$-planes.
The study of the $2$-dilation of maps between Riemannian manifolds has recently led to several interesting rigidity and topological applications; see, for example,
\cite{savas5,savas4,tsai,lee1,lee2,ding1,ding2}.

The first main result of our paper is the following one. In a certain sense, the theorem can be considered as an analogue of the well-known pinching property for the second fundamental form of minimal hypersurfaces in spheres due to Simons \cite{simons}, Lawson \cite{lawson} and Chern, Do Carmo and Kobayashi \cite{CdCK}.

\begin{mythm}\label{thmA}
Suppose that $f: \S^3\to\S^2$ is a harmonic map. Then the following hold:
\begin{enumerate}[\rm(1)]
\item
If the squared norm of its Hessian $B$ satisfies
\begin{equation}\label{eq_beautiful}
|B|^2 \le \operatorname{D}_2(\operatorname{D}_2+a),\quad\text{for some}\,\,\,a\in[0,2),
\end{equation}
then $f$ is the composition of a Hopf fibration with a holomorphic map $g: \S^2\to\S^2$ (up to
conjugation).
\smallskip
\item
If the squared norm of its Hessian $B$ satisfies
$$
|B|^2 < \operatorname{D}_2(\operatorname{D}_2+2),
$$
then $f$ is the composition of a Hopf fibration with a M\"obius transformation $g: \S^2\to\S^2$
(up to conjugation).
\end{enumerate}
\end{mythm}

\begin{remark}
Seeking for a close analogy with the case of minimal hypersurfaces, one may presume that the Hopf fibration plays a role similar to that of the Clifford torus. However, this doesn't seem to be the case for Theorem \ref{thmA}, as the Hopf fibration satisfies \eqref{eq_beautiful} with \emph{strict} inequality for any $a>0$, being $|B|^2 = 16$ and $\operatorname{D}_2 = 4$.
\end{remark}

For harmonic maps with totally geodesic fibers, this result can be improved to:

\begin{mythm}\label{thmC}
Suppose that $f: \S^3\to\S^2$ is a harmonic map with totally geodesic fibers. If the squared norm of its Hessian $B$ satisfies
$$
|B|^2\le 2\operatorname{D}_2(\operatorname{D}_2+2),
$$
then $f$ is the composition of a Hopf fibration with a holomorphic map $g: \S^2\to\S^2$ (up to
conjugation).
\end{mythm}

\begin{remark}
	From a technical viewpoint, there is a significant difference with the case of minimal hypersurfaces. If $f : M^n \to \S^{n+1}$ is a minimal hypersurface whose second fundamental form
	$B$ satisfies $|B|^2 \le n$, the rigidity result in \cite{simons,lawson,CdCK} uses the maximum principle applied to Simons' formula   
	\[
	\frac{1}{2} \Delta |B|^2 = |\nabla B|^2 + n|B|^2 - |B|^4.
	\]
	The Clifford torus is obtained in the case $|B|^2 = n$ by looking at minimal hypersurfaces satisfying $|\nabla B|^2 = 0$. In our setting, for harmonic maps $f : \S^3 \to \S^2$, one can compute the Laplacian of $|B|^2$ and obtain a more complicated formula of the type
	\[
	\Delta |B|^2 =2 |\nabla B|^2 \ + \ (\text{terms of order zero in $B$}).
	\]
	However, if $f$ is the Hopf fibration then $|\nabla B|^2 = 96$, hence discarding the term one loses much information. Furthermore, the refined Kato inequality for harmonic maps 
	\[
	|\nabla B|^2 \ge \frac{3}{2} \big| \nabla |B| \big|^2
	\]
	is of no help, since for the Hopf fibration the right-hand side is zero. For these reasons, our pinching Theorems \ref{thmA} and \ref{thmC} will instead be obtained by applying maximum principles to carefully chosen combinations of the singular values of $f$, without differentiating $B$.  
\end{remark}

In our last main result, we consider harmonic maps whose singular values satisfy a suitable identity, without imposing bounds on $B$. In particular, the theorem applies to characterize harmonic maps \( f : \mathbb{S}^{3} \to \mathbb{S}^{2} \) with constant energy density $|df|^2$ or constant 2-dilation $\operatorname{D}_2$. To state the result, recall that a function $f: C \subseteq \R^n \to \R$ is said to be $C^1$ on a subset $C$ if it can be extended to a $C^1$-function on an open set containing $C$.

\begin{mythm}\label{thmB}
Suppose that \( f : \mathbb{S}^{3} \to \mathbb{S}^{2} \) is a non-constant harmonic map whose two largest
singular values \( \lambda_{1}\ge \lambda_{2} \) satisfy
${W}(\lambda_1,\lambda_2)= 0,$ for some function $W$ which is continuous on the set  
\[
V = \big\{ (x,y) \in \R^2 \ : \ x \ge y \ge 0 \},
\]
of class $C^1$ on $V \cap \{x > y\} \cap W^{-1}(0)$ and there it satisfies 
\begin{equation}\label{eq_Weingarten}
y \frac{\partial {W}}{\partial x} + x \frac{\partial {W}}{\partial y} \neq 0, \quad \text{for all} \,\,\, (x,y) \in V \cap \{x > y\} \cap W^{-1}(0)
\end{equation}
Then, $f$ is the composition of a Hopf fibration with a holomorphic map $g: \S^2\to\S^2$ (up to
conjugation).
\end{mythm}

\begin{mythm}\label{thmD}
	Suppose that \( f : \mathbb{S}^{3} \to \mathbb{S}^{2} \) is a non-constant harmonic map with constant energy density or constant $2$-dilation. Then, $f$ is a Hopf fibration.
\end{mythm}

Let us conclude by pointing out some additional characterizations of the Hopf fibration within certain classes of maps from $\S^3$ to $\S^2$. We highlight the following three:
\begin{itemize}
\item Due to a result by Escobales \cite{escobales}, the (scaled) Hopf fibration is the only 
Riemannian submersion from $\S^3$ to $\S^2(1/2)$ up to isometry.
\smallskip
\item DeTurck, Gluck and Storm \cite{Turck} proved that a map $f: \S^3\to\S^2$ with non-zero Hopf invariant and singular values less than or equal to $2$ must
coincide with a Hopf fibration. In particular, up to isometries of domain and range, the Hopf fibration is the unique Lipschitz constant minimizer in its homotopy class.
\smallskip
\item The Hopf fibration has $2$-dilation $\operatorname{D}_2=4$. Recently, Guth and Lee \cite{lee} proved that
the Hopf invariant of a map $f:\S^3\to\S^2$ satisfies
$$
|{\rm Hopf}(f)|\le\frac{1}{16}\max_{\S^3} \operatorname{D}_2^2.
$$
Consequently, if $\max_{\S^3} \operatorname{D}_2<4$ then
$f$ is homotopically trivial. 
\end{itemize}

\noindent \textbf{Acknowledgements.} L.M. is supported by the PRIN project no. 20225J97H5 ``Differential-geometric aspects of manifolds via Global Analysis''.

\noindent \textbf{Conflict of Interests.} The authors have no conflict of interest.

\noindent \textbf{Data availability statement.} No data was generated by this research.

\section{Basics on maps from $\S^3$ to $\S^2$}
Let \(f : U \subset \mathbb{S}^3 \to \mathbb{S}^2\) be a smooth map, \(U\) being an open subset of \(\mathbb{S}^3\). 
The Euclidean metric on \(\mathbb{R}^4\) is denoted by \(\langle \cdot\,,\cdot \rangle\), and the same notation is used for the induced metrics on \(\mathbb{S}^3\) and \(\mathbb{S}^2\). 
Let \(D\) be the Levi-Civita connection of \(\mathbb{R}^4\), and \(\nabla\) the Levi-Civita connections of \(\mathbb{S}^3\) and \(\mathbb{S}^2\). We also use the same symbol $\nabla$ for the pullback connection on the bundle \(f^{*}T\mathbb{S}^2\). Denote by 
\[
\mathcal{V} = \ker df \quad\text{and}\quad \mathcal{H} = \mathcal{V}^\perp, 
\]
respectively, the {\em vertical} and {\em horizontal} distributions.

\subsection{The singular value decomposition}
Denote by
$$\lambda_1\ge\lambda_2\ge\lambda_3$$
the {\em singular values} of $df$, i.e. the eigenvalues of $(df^*df)^{1/2}$,
where $df^*$ is the transpose of $df$ with respect to the given metrics. Note that the singular values are continuous functions and that $\lambda_3\equiv 0$
for dimensional reasons. Suppose at the moment that on $U$
the kernel $\mathcal{V}$ of $df$ is a line bundle. 
Hence, in $U$ the two largest singular values of $df$ are non-zero. For simplicity, let us denote them by
$$\lambda\doteq\lambda_1\ge\lambda_2\doteq\mu>0.$$
At points where
$$\lambda>\mu>0$$
the singular values are even smooth, as the corresponding eigendistributions of $df^* df$ locally have constant dimensions, see \cite{nomizu}.
Fix $p \in U$ and consider an orthonormal basis $\{\alpha_{1}, \alpha_{2}, \alpha_{3}\}$ of $T_p\S^3$
and an orthonormal basis $\{\beta_{4}, \beta_{5}\}$ of $T_{f(p)}\S^2$
such that
\begin{equation}\label{singular}
\begin{array}{ccc}
d f(\alpha_{1})=\lambda\beta_4, & d f(\alpha_{2})=\mu \beta_{5}, & d f(\alpha_{3})=0.
\end{array}
\end{equation}
The map $f: U\subset\S^3\to\S^2$ is called:
\begin{enumerate}[\rm(1)]
\item {\em Horizontally conformal} if $df: \mathcal{H}\to T\S^2$ is conformal or, equivalently, if and only if its singular values satisfy
$\lambda\equiv \mu>0$ everywhere on $U$.
\smallskip
\item {\em Weakly conformal} if its two largest singular values satisfy
$\lambda\equiv \mu\ge 0$, everywhere on $U$.
\end{enumerate}
Note that in the case of weakly conformal maps we allow points where the differential of $f$ vanishes completely.

There are several quantities which encode information about how far $f: U\subset\S^3\to\S^2$ is from being
weakly conformal.
Consider the squared norm $u$ of the differential of $f$. In terms of the singular values of $f$
the function $u$ has the form
\begin{equation}\label{energy}
u\doteq |df|^2=\lambda^2+\mu^2.
\end{equation}
The function $u$ is also known in the literature as (twice) the {\em energy density} of $f$. Since $\mathcal{H}$ is
2-dimensional, it possesses a complex structure $J_{\mathcal{H}}$. Let
$\omega_{\S^2}$ be the volume form of $\S^2$. One can readily check that
the $2$-form $f^*\omega_{\S^2}$ is non-zero only on $\mathcal{H}$. Consequently, it must be a multiple of the volume
form $\omega_{\mathcal{H}}$ of $\mathcal{H}$. This means that there exists a smooth function $v$ on $U$ such that
\begin{equation}\label{2dil}
f^*\omega_{\S^2}\doteq v\,\omega_{\mathcal{H}}.
\end{equation}
In terms of the singular values of $f$ we have that
\begin{equation}\label{vdef}
|v|=\operatorname{D}_2=\lambda\mu,
\end{equation}
whence $v$ extends continuously to the entire $\mathbb{S}^3$ by setting $v= 0$ at points where the dimension of $\mathcal{H}$ is
less than 2. Without loss of generality we may assume that $v$ is positive on $U$.
Observe that
$$w\doteq u-2|v|=(\lambda-\mu)^2$$
is a non-negative function. In particular, if at some point $p_0\in U$ it holds that
$w(p_0)=0,$ then
$\lambda(p_0)=\mu(p_0).$
Hence $w$ measures how much $f$ deviates from being weakly conformal.

\subsection{The Hessian and the tension field}
The {\em Hessian} of $f: \S^3\to\S^2$ is defined to be the symmetric vector valued tensor $B$ given by
$$
B(X,Y)\doteq\nabla_{X}df(Y)-df(\nabla_XY),\quad\text{for}\,\,\, X,Y\in\mathfrak{X}(\S^3),
$$
where the first $\nabla$ stands for the natural connection on the pull-back bundle $f^*T\S^2$.
By straightforward computations, it follows that
\begin{equation}\label{codazzi}
\begin{array}{lcl}	
\disp (\nabla_XB)(Y,Z)-(\nabla_ZB)(X,Y)
& = & \disp \big\{\langle X,Y\rangle-\langle df(X),df(Y)\rangle \big\}df(Z) \\[0.3cm]
& & \disp -\big\{\langle Y,Z\rangle-\langle df(Y),df(Z)\rangle \big\}df(X),
\end{array}
\end{equation}
for all $X,Y,Z\in\mathfrak{X}(\S^3)$.
Equation \eqref{codazzi} will be referred to as the {\em Codazzi equation}.
The trace of $B$
is called the {\em tension field} $\tau(f)$ of $f$. If $\tau(f)$ is identically zero the map $f$
is called {\em harmonic}. It is well-known that a harmonic map is real-analytic.
Throughout the paper we adopt the following terminology:
\begin{equation*}
B_{ij}=B(\alpha_i,\alpha_j), \qquad b_{ij}^{c}\doteq B^{c}(\alpha_i,\alpha_j)\doteq\langle B(\alpha_{i}, \alpha_{j}), \beta_{c}\rangle,
\end{equation*}
where $i,j,k \in \{1,2,3\}$ and $c \in \{4,5\}$. Moreover, at points where the frame $\{\alpha_1,\alpha_2,\alpha_3\}$ is smooth we
set
$$
\omega_{ij}(\alpha_k)\doteq\langle\nabla_{\alpha_k}\alpha_i,\alpha_j\rangle.
$$

\section{The Hopf fibration and harmonic unit vector fields}

\subsection{The Hopf fibration}\label{sec_hopf}

Let us briefly describe the Hopf fibration,
following the exposition of Pinkall \cite{pinkall}. Identify $\R^4$ with the {\em quaternions}
$$
\mathbb{H} \doteq\{x_0 + x_1  i + x_2 j + x_3 k : x_0, x_1, x_2, x_3\in\R\},
$$
which form an associative algebra with $1$ as the multiplicative unit via
$$
i^2=j^2=k^2=ijk=-1.
$$
Denote by $\Re \mathbb{H}$ the $1$-dimensional linear subspace spanned by the element $1$, and
by $\Im \mathbb{H}$ the orthogonal complement of $\Re \mathbb{H}$. Given an element
$$x=x_0+x_1i+x_2j+x_3k$$
we define its {\em conjugate} as the element
$$
\overline{x}\doteq x_0-x_1i-x_2j-x_3k.
$$
Multiplication by $i$, $j$ and $k$ gives rise to complex structures $J_1$, $J_2$ and $J_3$ on $\mathbb{H}$ by
$$
J_1x\doteq i\cdot x,\quad J_2x\doteq j\cdot x\quad\text{and}\quad J_3x\doteq k\cdot x,\quad\text{for all}\,\,\, x\in\mathbb{H}.
$$
Consider the unit sphere $\S^{3}\subset \mathbb{H}$ as the subset of quaternions of length $1$,
and identify $\S^2$ with  the unit sphere in the subspace of $\mathbb{H}$ spanned by $1$, $j$ and $k$.
Let $\tilde\cdot$  denote  the anti-automorphism that fixes $1$, $j$ and $k$ and sends $i$ to $-i$, i.e.
$$
x_0+x_1i+x_2j+x_3k=x\mapsto\tilde{x}=x_0-x_1i+x_2j+x_3k.
$$
\begin{lemma}\label{hopfsmap}
	Let $\pi: \S^3\to\mathbb{H}$ be the smooth map given by
	$
	\pi(p)\doteq\tilde{p}\cdot p$, for each $p\in\S^3.$
	Then:
	\begin{enumerate}[\rm(1)]
		\item If $(x_0,x_1,x_2,x_3)\in\S^3\subset\mathbb{H}$, the map $\pi$ has the form
		$$\pi(x_{0},x_{1},x_{2},x_{3})=(x_{0}^{2}+x_{1}^{2}-x_{2}^{2}-x_{3}^{2},2(x_{0}x_{2}+x_{1}x_{3}),2(x_{0}x_{3}-x_{1}x_{2})).$$
		\item The map $\pi: \S^3\to\S^2$ is a surjective submersion.
		\medskip
		\item $\pi((\cos\theta+i\sin\theta)p)=\pi(p)$, for $\theta\in[0,2\pi]$ and $p\in\S^3$.
		\medskip
		\item $\{\alpha_1(p)=j\cdot p,\alpha_2(p)=k\cdot p,\alpha_3=i\cdot p\}$ is an
		orthonormal frame on $\S^3$.
		\medskip
		\item
		$\mathcal{V}={\rm ker}\,d\pi=\span\{\alpha_3\}$ and $\mathcal{H}=\mathcal{V}^{\perp}=\span\{\alpha_1,\alpha_2\}.$ In particular,
		the integral curves of $\alpha_3$ are great circles.
		\medskip
		\item
		The singular values of $\pi$ are constant; namely $\lambda_1=2=\lambda_2$ and $\lambda_3=0$.
		\medskip
		\item The map $f\doteq(1/2)\pi: \S^3\to\S^2(1/2)$ is a Riemannian submersion.
	\end{enumerate}
\end{lemma}

The map $f$ described in the lemma above is known as the {\em standard Hopf fibration} and the associated vector
field $\alpha_3=i\cdot p$ as the {\em standard Hopf vector field}. Any composition of $f$ with
isometries or dilations of the corresponding spheres is called a {\em Hopf fibration}. Furthermore, any unit vector field
$\zeta$ on $\S^3$ that can be expressed in the form
$$\S^3\ni p\mapsto\zeta_p=J(p)\in T_p\S^3,$$
where $J$ is a complex structure on $\R^4$,
is referred to in the literature as a {\em Hopf vector field}. Observe that the Hopf fibration is precisely the quotient map associated with the great circle fibration generated by the corresponding Hopf vector field.

In the next lemma we collect some further properties of a Hopf fibration.

\begin{lemma}
	Let $\pi: \S^{3}\to\S^2$ be the map given in Lemma \ref{hopfsmap} and $\{\alpha_1,\alpha_2,\alpha_3=\zeta\}$ the globally defined
	orthonormal frame on $\S^3$ arising from the singular value decomposition of $\pi$. Then:
	\begin{enumerate}[\rm(1)]
		\item
		The $(1,1)$-tensor $\varphi$ given by
		$$
		\varphi(X)\doteq-\nabla_X\zeta
		$$
		is a complex structure on $\mathcal{H}$. Moreover, $\zeta$ satisfies the differential equation
		\begin{equation*}
			\Delta\zeta+|\nabla\zeta|^2\zeta=0,
		\end{equation*}
		where $\Delta$ is the rough Laplacian on vector fields.
		\medskip
		\item
		The map $\pi$ is harmonic, and the components of its Hessian with respect to the bases $\{\alpha_1,\alpha_2,\alpha_3;\beta_4,\beta_5\}$ arising from the singular value decomposition of $\pi$ are given by
		$$
		B^{4}=
		\left( {\begin{array}{rrr}
				0 & 0 & 0\\
				0 & 0 & -2 \\
				0 & -2 & 0\\
		\end{array} } \right)
		\quad\text{and}\quad
		B^{5}=
		\left( {\begin{array}{rrr}
				\,\,0 & 0 & 2\\
				0 &\,\,0 & 0 \\
				\,\,2 & \,\,0 & 0\\
		\end{array} } \right),
		$$
		respectively. Moreover
		$$|B|^2=16\quad\text{and}\quad|\nabla B|^2=6|B|^2=96.$$
		\item Suppose that $g: \S^2\to\S^2$ is a conformal map. Then the composition $h=g\circ\pi$ is again
		a harmonic map with totally geodesic fibers. Moreover, the Hessian $B_h$ of $h$ is given by
		$$
		B^{4}_h=2
		\left( {\begin{array}{ccr}
				2\beta_4(\sigma) & \,\,\,\,2\beta_5(\sigma) & 0\\
				2\beta_5(\sigma) & -2\beta_4(\sigma) & \sigma \\
				0 & \sigma & 0\\
		\end{array} } \right)
		\quad\text{and}\quad
		B^{5}_h=2
		\left( {\begin{array}{ccr}
				-2\beta_5(\sigma) & \,\,\,\,2\beta_4(\sigma) & -\sigma\\
				\,\,\,\,2\beta_4(\sigma) &\,\,\,\,2\beta_5(\sigma) & 0 \\
				-\sigma & 0 & 0\\
		\end{array} } \right),
		$$
		where $\sigma$ is the (double) singular value of $g$.
		Moreover,
		\begin{equation}\label{estbh1}
			|B_h|^2=16(\sigma^2+4|\nabla\sigma|^2).
		\end{equation}
		If $g$ is neither constant nor an isometry, then
		\begin{equation}\label{estbh2}
			\max|B_h|^2>16.
		\end{equation}
	\end{enumerate}
\end{lemma}
\begin{proof}
	(1) Since the integral curves of $\zeta$ are great circles it follows that
	$$\varphi(\zeta)=0.$$
	Moreover,
	from the Weingarten formula, we have that
	\begin{eqnarray*}
		\varphi(\alpha_1)=-\nabla_{\alpha_1}\zeta
		=-D_{\alpha_1}(J_1x)=-J_1D_{\alpha_1}x
		=-J_1\alpha_1=-J_1J_2x=-J_3x=-\alpha_2.
	\end{eqnarray*}
	Similarly,
	$$
	\varphi(\alpha_2)=\alpha_1.
	$$
	The formula for the Laplacian of $\zeta$ follows by straightforward computations.
	
	(2) By the Koszul formula and the fact that
	$\pi$ has constant singular values we deduce that
	\begin{equation}\label{B1}
		B(X,Y)=\nabla_{X}d\pi(Y)-d\pi(\nabla_{X}Y)=0, \quad\text{for all}\,\, X,Y\in\Gamma(\mathcal{H})
	\end{equation}
	for details see also \cite[Lemma 4.5.1, page 119]{baird1}.
	Hence, $B$ vanishes on the horizontal bundle. Since integral
	curves of $\zeta$ are geodesics and it belongs to the kernel of $d\pi$, we obtain that
	\begin{equation}\label{B4}
		B(\zeta,\zeta)=0.
	\end{equation}
	and
	\begin{equation}\label{B5}
		B(\zeta,X)=-d\pi(\nabla_{X}\zeta)=d\pi(\varphi(X)),\quad\text{for any}\,\, X\in\Gamma(\mathcal{H}).
	\end{equation}
	Differentiating \eqref{B1}, \eqref{B4} and \eqref{B5} we get the estimates for $|B|$ and
	$|\nabla B|$.
	
	(3) Let us denote by $\sigma$ the conformal factor of $g$.  By a straightforward computation, we get
	\begin{equation}\label{Bcompose}
		B_{h}(X,Y)=B_g(d\pi(X),d\pi(Y))+dg(B(X,Y)),
	\end{equation}
	for any vector fields $X$ and $Y$ of the sphere $\S^{3}$. Hence,
	$$
	B_h(\zeta,\zeta)=B_g(d\pi(\zeta),d\pi(\zeta))+dg(B(\zeta,\zeta))=0.
	$$
	Consider at $h(x)$ an orthonormal frame
	$\{{\tilde\beta_4},{\tilde\beta}_5\}$ such that
	$$
	dg(\beta_4)=\sigma{\tilde\beta}_4\quad\text{and}\quad dg(\beta_5)=\sigma{\tilde\beta}_5.
	$$
	Again by  Koszul's formula we get, away from the zero set of $\sigma$, that
	\begin{eqnarray*}
		B_g(X,Y)=X(\log\sigma)dg(Y)
		+Y(\log\sigma)dg(X)-\langle X,Y\rangle dg(\nabla \log\sigma),
	\end{eqnarray*}
	where $X$ and $Y$ are tangent vectors of $\S^2$. Hence,
	from \eqref{Bcompose}, \eqref{B1}, \eqref{B4} and \eqref{B5} we deduce that
	$$
	B_g(\alpha_3,\alpha_3)=0,\quad B_g(\alpha_1,\alpha_3)=2\sigma{\tilde\beta}_5\quad\text{and}\quad
	B_g(\alpha_2,\alpha_3)=-2\sigma{\tilde\beta}_4.
	$$
	Furthermore
	$$
	B_h(\alpha_1,\alpha_1)=4\beta_4(\sigma)\tilde\beta_4-4\beta_5(\sigma){\tilde\beta}_5
	=-B_h(\alpha_2,\alpha_2)
	$$
	and
	$$
	B_h(\alpha_1,\alpha_2)=4\beta_5(\sigma)\tilde\beta_4+4\beta_4(\sigma){\tilde\beta}_5.
	$$
	Assume now that $g$ is neither constant nor an isometry. We claim that
	$$
	\min\sigma^2<1<\max\sigma^2.
	$$
	Indeed, suppose to the contrary that $\max\sigma^2\le 1$, i.e. $g$ is a weakly
	length decreasing map. Because $g$ is
	conformal, its graph is a minimal submanifold
	of $\S^2\times\S^2$; see for example \cite{eells}. According to \cite[Theorem A]{savas4},
	$g$ must be constant or an isometry, a contradiction. Assume now that
	$\min\sigma^2\ge 1$. Then $g$ is a local diffeomorphism. Since $\S^2$ is
	simply connected, $g$ is a global diffeomorphism. Then $g^{-1}$ is a weakly
	length decreasing minimal diffeomorphism; see for example \cite{schoen2}. Again due to \cite[Theorem A]{savas4}, $g$ must be
	an isometry. Since $\nabla\sigma$ vanishes at maximum points of $\sigma$, it follows
	from \eqref{estbh1}
	that $|B_h|^2$ takes values above $16$.
	This completes the proof.
\end{proof}

\subsection{Harmonic unit vector fields on $\S^3$} In the previous section we proved that the
Hopf vector field $\zeta$ satisfies the equation
\begin{equation}\label{harmonicz}
\Delta\zeta+|\nabla\zeta|^2\zeta=0,
\end{equation}
where $\Delta$ is the rough Laplacian of $\zeta$. It turns out that \eqref{harmonicz} is the Euler-Lagrange equation of the {\em energy functional}
$$
E(\zeta)\doteq\frac{3}{2}{\rm vol}(\S^3)+\frac{1}{2}\int_{\S^3}|\nabla\zeta|^2d\mu
$$
in the space of unit vector fields on $\S^3$,
where $d\mu$ denotes the volume element and ${\rm vol}(\S^3)$ is the volume of $\S^3$;
see \cite{wiegmink,wood}.
Critical points of $E$ for variations by unit vector fields are called {\em harmonic
	unit vector fields}. It is conjectured that the Hopf vector fields
are the only harmonic unit vector fields. 
This conjecture remains open in general, but Fourtzis, Markellos, and Savas-Halilaj \cite[Theorem B \& C]{savas2} 
proved it under the additional assumption that $\zeta$ has totally geodesic integral leaves. To state the result, we first recall the following topological lemma.
\begin{lemma}\label{lem_topological}
	Let $\zeta$ be a unit vector field on $\S^3$ with compact integral leaves. Then, the quotient space obtained by identifying points on the same leaves of $\zeta$ is a smooth compact surface diffeomorphic to $\S^2$, and the quotient map $\pi : \S^3 \to \S^2$ is a smooth fiber bundle (in particular, a submersion).  
\end{lemma}
	\begin{proof}
	Let us denote by $\mathbb{M}^2$ the space of leaves of the foliation generated by the integral curves of $\zeta$, and denote by $\pi:\S^3\to\mathbb{M}^2$ the corresponding quotient map, i.e. the map which identifies points of the same integral curve of $\zeta$. Note that $\mathbb{M}^2 = \pi(\S^3)$ with the quotient topology is compact and connected. Since $\zeta$ has compact  leaves, the distribution it generates is regular in the sense of \cite[Chapter 1]{palais}, thus by \cite[Corollary 5 p.24]{palais} the triple
	$$
	\mathbb{S}^1\hookrightarrow \S^3\overset{\pi}{\longrightarrow}\mathbb{M}^2
	$$
	is $\S^1$ fiber bundle and $\mathbb{M}^2$ is a smooth $2$-dimensional manifold. Moreover, by the long exact homotopy sequence of a fiber bundle (see \cite[Theorem 4.41 and Proposition 4.48]{hatcher})
	\[
	\ldots \to \pi_1(\S^3) \to \pi_1(\mathbb{M}^2) \to \pi_0(\S^1) \to \pi_0(\S^3) \to 0
	\]
	we deduce that $\mathbb{M}^2$ is simply connected.
Therefore, $\mathbb{M}^2$ is diffeomorphic to the sphere $\S^2$.
	\end{proof}

\begin{theorem}[\cite{savas2}]\label{fms}
	Let $\zeta\in\mathfrak{X}(\S^3) $ be a unit vector field whose integral curves are
	great circles and let
	$\pi : \S^3 \to \S^2$ be the corresponding quotient map. Equip the spheres with their standard round metrics. Then:
	\begin{enumerate}[\rm(1)]
		\item The map $\pi$ is harmonic if and only if $\zeta$
		is a harmonic unit vector field.
		\medskip
		\item
		If $\zeta$ is globally defined, then it
		coincides with a Hopf vector field and the corresponding quotient map $\pi: \S^3\to\S^2$ is a Hopf fibration.
	\end{enumerate}
\end{theorem}

\section{Bochner-Weitzenb\"ock formulas}
In the sequel we will compute the gradients and Laplacians of the functions $u$, $v$ and $w$ defined in the
previous section, in the case where
$f$ is a harmonic map.

\begin{lemma}\label{7.1}
Let $f: U\subset\S^3 \to \S^2$ be a submersion and $k \in \{ 1,2,3 \}$.
On the open and dense subset of $U$ where $\lambda$ and $\mu$ are smooth we have that
\begin{equation*}\label{eq1}
  \alpha_{k}(\lambda) = b_{1k}^{4}\quad \text{and} \quad \alpha_{k}(\mu) = b_{2k}^{5}.
\end{equation*}
Moreover,
\begin{equation*}\label{eq2}
  \lambda\omega_{13}(\alpha_{k}) = b_{3k}^{4}\quad \text{and} \quad \mu\omega_{23}(\alpha_{k}) = b_{3k}^{5}
\end{equation*}
and
\begin{equation*}
  (\lambda^2 - \mu^2)\omega_{12}(\alpha_{k}) = \lambda b_{2k}^{4} + \mu b_{1k}^{5},
\end{equation*}
where the indices are with respect to the frames arising from the singular value decomposition.
\end{lemma}
\begin{proof}
Differentiating the equation $\lambda^2 = |d f(\alpha_{1})|^2$
with respect to $\alpha_k$ we get
$$
\lambda \alpha_{k}(\lambda)=\langle B(\alpha_{k}, \alpha_{1})+df(\nabla_{\alpha_{k}}\alpha_{1}), d f(\alpha_{1})\rangle
=\langle B(\alpha_{k}, \alpha_{1}), d f(\alpha_{1})\rangle=\lambda\, b^4_{1k}.
$$
In a similar manner, by differentiating the equation
$
\mu^2 = |d f(\alpha_{2})|^2,
$
we obtain that
$$\alpha_{k}(\mu) = b_{2k}^{5}.$$
On the other hand, because $\alpha_3$ is in the kernel of $df$, we get that
\begin{equation*}
  B(\alpha_{k},\alpha_3) = -df(\nabla_{\alpha_{k}}\alpha_3) = \lambda\omega_{13}(\alpha_{k})\beta_4 + \mu\omega_{23}(\alpha_{k})\beta_5.
\end{equation*}

Differentiating
$$\langle df(\alpha_1),df(\alpha_2)\rangle=0$$
with respect to $\alpha_k$ we get the expression for $\omega_{12}(\alpha_k)$.
This concludes the proof.
\end{proof}

\begin{lemma}\label{7.2}
  Let $f : U \subset \S^{3} \to \S^{2}$ be a harmonic map. On the open set where $\lambda$ is different from $\mu$, we have that $\lambda^2$ and $\mu^2$ are smooth functions, and their Laplacians are given by
  \[ \frac{1}{2}\Delta(\lambda^2) = (2-\mu^2)\lambda^2 + \sum_{k=1}^{3} \frac{(\lambda b_{2k}^{4} + \mu b_{1k}^{5})^2}{\lambda^2 - \mu^2} + \sum_{k=1}^{3} |B_{1k}|^2 + \sum_{k=1}^{3} (b_{3k}^{4})^2, \]
  and
  \[ \frac{1}{2}\Delta(\mu^2) = (2-\lambda^2)\mu^2 - \sum_{k=1}^{3} \frac{(\lambda b_{2k}^{4} + \mu b_{1k}^{5})^2}{\lambda^2 -\mu^2} + \sum_{k=1}^{3} |B_{2k}|^2 + \sum_{k=1}^{3} (b_{3k}^{5})^2,
  \]
where the indices are with respect to the frames arising from the singular value decomposition.
\end{lemma}

\begin{proof}
On the open set $V$ where $\lambda>\mu$ the multiplicity of $\lambda^2$ does not change. Thus
$\lambda^2$ is smooth \cite{nomizu}. Since $u$ is smooth on $\S^3$ it follows that $\mu^2$ is smooth
on $V$ as well.
  We have
  \begin{equation*}
    \begin{aligned}
      \frac{1}{2}\Delta(\lambda^2) &=\frac{1}{2} \sum_{k=1}^{3} (\nabla_{\alpha_{k}}\nabla_{\alpha_{k}} - \nabla_{\nabla_{\alpha_{k}}\alpha_{k}})\langle df(\alpha_1),df(\alpha_1) \rangle \\
      &= \sum_{k=1}^{3} \langle \nabla_{\alpha_{k}}B(\alpha_1,\alpha_{k}), df(\alpha_1) \rangle + \sum_{k=1}^{3} \langle B(\alpha_1,\alpha_{k}), B(\alpha_1,\alpha_{k}) + df(\nabla_{\alpha_{k}}\alpha_1) \rangle \\
      &\quad - \sum_{k=1}^{3} \langle B(\alpha_1,\nabla_{\alpha_k}\alpha_{k}),df(\alpha_1) \rangle \\
      &= \sum_{k=1}^{3} \langle (\nabla_{\alpha_{k}}B)(\alpha_1,\alpha_{k}) + B(\nabla_{\alpha_{k}}\alpha_1,\alpha_{k}),df(\alpha_1) \rangle + \sum_{k=1}^{3} \big( |B_{1k}|^2 + \omega_{12}(\alpha_{k})\mu b_{1k}^{5}\big).
    \end{aligned}
  \end{equation*}
  By the Codazzi equation \eqref{codazzi} and Lemma \ref{7.1} we get that
  \begin{equation*}
    \begin{aligned}
      \frac{1}{2}\Delta(\lambda^2) &= (2-\mu^2)\lambda^2 + \sum_{k=1}^{3} \frac{(\lambda b_{2k}^{4} + \mu b_{1k}^{5})^2}{\lambda^2 - \mu^2} + \sum_{k=1}^{3} |B_{1k}|^2 + \sum_{k=1}^{3} (b_{3k}^{4})^2.
    \end{aligned}
  \end{equation*}
In a similar manner
  \begin{equation*}
    \begin{aligned}
      \frac{1}{2}\Delta(\mu^2)
      &= \sum_{k=1}^{3} \langle \nabla_{\alpha_{k}}B(\alpha_2,\alpha_{k}), df(\alpha_2) \rangle + \sum_{k=1}^{3} \langle B(\alpha_2,\alpha_{k}), B(\alpha_2,\alpha_{k}) + df(\nabla_{\alpha_{k}}\alpha_2) \rangle \\
      &\quad - \sum_{k=1}^{3} \langle B(\alpha_2,\nabla_{\alpha_k}\alpha_{k}),df(\alpha_2) \rangle \\
      &= \sum_{k=1}^{3} \langle (\nabla_{\alpha_{K}}B)(\alpha_2,\alpha_{k}) + B(\nabla_{\alpha_{k}}\alpha_2,\alpha_{k}),df(\alpha_2) \rangle + \sum_{k=1}^{3} \big( |B_{2k}|^2 + \omega_{21}(\alpha_{k})\lambda b_{2k}^{4}) \big) \\
      &= (2-\lambda^2)\mu^2 - \sum_{k=1}^{3} \frac{(\lambda b_{2k}^{4} + \mu b_{1k}^{5})^2}{\lambda^2 -\mu^2} + \sum_{k=1}^{3} |B_{2k}|^2 + \sum_{k=1}^{3} (b_{3k}^{5})^2,
    \end{aligned}
  \end{equation*}
  which completes the proof.
\end{proof}

Making use of Lemmata \ref{7.1} and \ref{7.2}, we obtain the following by direct computations.

\begin{lemma}\label{lapu}
Let $f: U\subset\S^{3} \to \S^{2}$ be a harmonic map. The gradient and the Laplacian of the function $u$
defined in \eqref{energy}
are given by
\begin{equation}\label{simera}
  \nabla u = 2\sum_{k=1}^{3} (\lambda b_{1k}^{4} + \mu b_{2k}^{5}) \alpha_{k}
\end{equation}
where the indices are with respect to the frames of the singular value decomposition. Moreover
$$
\Delta u =2|B|^2+4(u-v^2),
$$
where $v$ is the function given in \eqref{vdef}.
In particular $v^2$ is a real analytic function on $U$.
\end{lemma}

\begin{lemma}\label{lapv}
Let $f: U\subset\S^{3} \to \S^{2}$ be a harmonic submersion. The gradient and Laplacian of the product $\lambda\mu$ satisfy
$$
\nabla (\lambda\mu) = \sum_{k=1}^{3} (\mu b_{1k}^{4} + \lambda b_{2k}^{5}) \alpha_{k}
$$
and
$$
\Delta (\lambda\mu) = \lambda\mu(4-\lambda^2-\mu^2) + 2\sum_{k=1}^{3} (b_{1k}^{4}b_{2k}^{5} - b_{1k}^{5}b_{2k}^{4}) + \frac{\mu}{\lambda}\sum_{k=1}^{3} (b_{3k}^{4})^2 + \frac{\lambda}{\mu} \sum_{k=1}^{3} (b_{3k}^{5})^2,
$$
where the indices are with respect to the frames arising from the singular value decomposition.
\end{lemma}

As an immediate consequence of Lemmata~\ref{lapu} and \ref{lapv} we obtain the following result.
\begin{lemma}\label{lapw}
Let $f: U\subset\S^{3} \to \S^{2}$ be a harmonic submersion. Consider the function
$$w=(\lambda-\mu)^2\ge 0.$$
At points where $w$ is smooth, its gradient and Laplacian are given by
$$
\nabla w = 2(\lambda - \mu)\sum_{k=1}^{3} (b_{1k}^{4} - b_{2k}^{5}) \alpha_{k}
$$
and
$$
\frac{1}{2}\Delta w = w(\lambda\mu+2) + | B |^{2} - 2\sum_{k=1}^{3} (b_{1k}^{4}b_{2k}^{5} - b_{1k}^{5}b_{2k}^{4}) - \frac{\mu}{\lambda}\sum_{k=1}^{3} (b_{3k}^{4})^2 - \frac{\lambda}{\mu} \sum_{k=1}^{3} (b_{3k}^{5})^2,
$$
where the indices are with respect to the frames arising from the singular value decomposition.
\end{lemma}

Moreover, by a direct computation we obtain the following lemma.

\begin{lemma}\label{lapbarw}
Let $f: U\subset\S^{3} \to \S^{2}$ be a harmonic map.
Consider the function
$$\varrho\doteq \sqrt{u^2-4v^2}=\lambda^2-\mu^2\ge 0.$$
At points where $\varrho$ is smooth, its Laplacian is given by
\begin{equation*}
  \frac{1}{2}\Delta\rho = 2\rho + 2\rho^{-1} \sum_{k=1}^{3} (\lambda b_{2k}^{4} + \mu b_{1k}^{5})^2 + |B_{11}|^2 - |B_{22}|^2 + 2(b_{13}^{4})^2 - 2(b_{23}^{5})^2 + (b_{33}^{4})^2 - (b_{33}^{5})^2,
\end{equation*}
where the indices are with respect to the frames arising from the singular value decomposition.
\end{lemma}

\section{Weakly conformal harmonic maps from $\S^3$ to $\S^2$}\label{hopfmap}

In this section we consider weakly conformal harmonic maps $f: U\subset\S^3\to\S^2$. If $U$ is the entire $\S^3$, such maps were classified by Baird and Wood in \cite[Theorem 5.1]{baird2}. They showed the following:

\begin{theorem}\label{bairdwood}
Suppose that $f : \mathbb{S}^3 \to \mathbb{S}^2$ is a weakly conformal harmonic map. 
Then $f$ factors as the composition of a Hopf fibration
$\pi : \mathbb{S}^3 \to \mathbb{S}^2$ with a conformal map $g : \mathbb{S}^2 \to \mathbb{S}^2$. 
In particular, if $f$ is a submersion, then $g$ is a M\"obius transformation (up to conjugation).
\end{theorem}

For the sake of completeness, we include a proof of this result which combines Theorem \ref{fms} with a detailed 
analysis of the local structure of the locus where the differential of the map $f$ vanishes identically.

We first recall a couple of useful facts about weakly conformal submersions. Assume that $U \subset \S^3$ is an open subset and that $f : U \to f(U) \subset \S^2$ is a submersion. Then there exists a unique
unit vector field $\zeta\in\mathfrak{X}(U)$ spanning $\ker(df)$ and such that $f^*\omega_{\S^2}\wedge \zeta_{\flat}$ is in the orientation of $\S^3$.

The following lemma is well-known in the literature, see for example \cite[Proposition 4.5.3]{baird1}, but for the sake of completeness we include a short proof.

\begin{lemma}\label{harmonic iff totally geodesic fibers}
Let $f: U\subset\S^3\to\S^2$ be a weakly conformal submersion, defined in an open neighbourhood $U$ of $\S^3$.
Then $f$ is harmonic if and only if its fibers are totally geodesic.
\end{lemma}

\begin{proof}
We define $\zeta$ as above on $U$, and denote by $\lambda$ the non-zero (double) singular value of $df$. By Koszul's formula we have that
\begin{equation}\label{eqkoz1}
B(X,Y)=X(\log\lambda)df(Y)+Y(\log\lambda)df(X)-\langle X,Y\rangle df(\nabla\log\lambda),
\end{equation}
for any pair of horizontal vector fields $X,Y\in\Gamma(\mathcal{H})$. Furthermore,
\begin{equation}\label{eqkoz2}
B(\zeta,\zeta)=\nabla_{\zeta}df(\zeta)-df(\nabla_{\zeta}\zeta)=-df(\nabla_{\zeta}\zeta)=-df(H_{\zeta}),
\end{equation}
where $H_{\zeta}$ is the geodesic curvature of the integral curves of $\zeta$. From
\eqref{eqkoz1} and \eqref{eqkoz2} we readily see that
$$
\tau(f)=-df(H_{\zeta}).
$$
This completes the proof.
\end{proof}

We now show that $\zeta$ is in fact a harmonic unit vector field.

%{\color{blue}In the next lemma, it is not clear to me if $\zeta$ can be consistently defined on the entire $U$. The assertion seems to me nontrivial. One way to prove it is to proceed as follow (but it is better to put it at the end of our proof): at every point $p \in U$ we define a line bundle $L$ by associating $\ker(df_p)$. We prove that such a line bundle extends continuously to the entire $\S^3$. Since $\S^3$ has zero Euler characteristic (careful: we need that the Euler characteristic of the bundle $L$ is zero to find a section, right?), there exists a nowhere vanishing continuous section $\zeta$, which by normalization we can assume to have unit norm at every point. Locally in $U$, this accounts for one of the two choices of the generators of $\ker(df)$, thus $\zeta$ is smooth (and harmonic) on $U$ and continuous on the entire $\S^3$. By the removable singularity theorem, $\zeta$ is smooth and harmonic on the entire $\S^3$.}\\ 

\begin{lemma}\label{lem_geofibers}
Let $f: U\subset\S^3\to\S^2$ be a weakly conformal harmonic submersion. Then the unit vector field $\zeta$ generating the vertical distribution $\mathcal{V}$ is harmonic.
\end{lemma}

\begin{proof}
Consider a local orthonormal frame $\{\alpha_1,\alpha_2,\alpha_3\doteq\zeta\}$ and denote by $\lambda$ the non-zero (double) singular value of $df$.
From Lemma \ref{7.1} for $\lambda\equiv \mu$ we see that
\begin{equation}\label{tanos1}
b_{1k}^4=b_{2k}^5\quad\text{and}\quad b_{1k}^5=-b_{2k}^4,\quad\text{for}\,\, k\in\{1,2,3\}.
\end{equation}
Define the vector field
$$
Z\doteq \Delta\zeta+|\nabla \zeta|^2\zeta.
$$
We can easily see that $Z$ is perpendicular to $\zeta$. Indeed
\[ \langle Z,\alpha_3\rangle = \sum_{k=1}^{2}\langle \nabla_{\alpha_k}\nabla_{\alpha_k}\alpha_3,\alpha_3\rangle +\sum_{k=1}^2|\nabla_{\alpha_k}\alpha_3|^2 = -\sum_{k=1}^{2}\langle \nabla_{\alpha_k}\alpha_3,\nabla_{\alpha_k}\alpha_3\rangle +\sum_{k=1}^2|\nabla_{\alpha_k}\alpha_3|^2 = 0. \]
It remains to show that the horizontal component of $Z$ is zero as well, or equivalently
that
$$df(Z)=0.$$
Using the Codazzi equation and the harmonicity of $f$, we obtain
$$
\begin{aligned}
df(Z) &=\sum_{k=1}^3 df\big\{\nabla_{\alpha_{k}}\nabla_{\alpha_{k}}\alpha_3 - \nabla_{\nabla_{\alpha_{k}}\alpha_{k}}\alpha_3\big\} \\
&=\sum_{k=1}^3 \big\{\nabla_{\alpha_{k}}df(\nabla_{\alpha_{k}}\alpha_3) - B(\alpha_{k},\nabla_{\alpha_{k}}\alpha_3) - \nabla_{\nabla_{\alpha_{k}}\alpha_{k}}df(\alpha_3) + B(\nabla_{\alpha_{k}}\alpha_{k},\alpha_3) \big\}\\
&=\sum_{k=1}^3 \big\{\nabla_{\alpha_{k}}\nabla_{\alpha_{k}}df(\alpha_3) - \nabla_{\alpha_{k}}B(\alpha_{k},\alpha_3) -B(\alpha_{k},\nabla_{\alpha_{k}}\alpha_3) + B(\nabla_{\alpha_{k}}\alpha_{k},\alpha_3)\big\}. \\
&=-\sum_{k=1}^3  \big\{(\nabla_{\alpha_{k}}B)(\alpha_3,\alpha_{k}) -2B(\nabla_{\alpha_{k}}\alpha_3, \alpha_{k})
\big\} \\
&=-\sum_{k=1}^3  \big\{(\nabla_{\alpha_3}B)(\alpha_{k},\alpha_{k}) + 2\omega_{13}(\alpha_{k})B_{1k} + 2\omega_{23}(\alpha_{k})B_{2k}\big\}\\
&=-\sum_{k=1}^3  \big\{2\omega_{13}(\alpha_{k})B_{1k} + 2\omega_{23}(\alpha_{k})B_{2k}\big\}.
\end{aligned}
$$
By Lemma \ref{7.1} for $\lambda\equiv\mu$ and \eqref{tanos1} we deduce that
$$
\begin{aligned}
df(Z)&= \frac{2}{\lambda}\sum_{k=1}^3(b^{4}_{k3}B_{1k} + b^{5}_{k3}B_{2k}) \\
&= \frac{2}{\lambda}(b_{13}^{4}b_{11}^{4} + b_{23}^{4}b_{12}^{4} + b_{13}^{5}b_{12}^{4} + b_{23}^{5}b_{22}^{4})\beta_{4} \\
&\quad+ \frac{2}{\lambda}(b_{13}^{4}b_{11}^{5} + b_{23}^{4}b_{12}^{5} + b_{13}^{5}b_{12}^{5} + b_{23}^{5}b_{22}^{5})\beta_{5} \\
&= \frac{2}{\lambda}(b_{13}^{4}b_{11}^{4} + b_{23}^{4}b_{12}^{4} - b_{23}^{4}b_{12}^{4} - b_{13}^{4}b_{11}^{4})\beta_{4} \\
&\quad+ \frac{2}{\lambda}(b_{13}^{4}b_{11}^{5} + b_{23}^{4}b_{12}^{5} - b_{23}^{4}b_{12}^{5} - b_{13}^{4}b_{11}^{5})\beta_{5} \\
&= 0,
\end{aligned}
$$
which concludes the proof.
\end{proof}

We are ready to prove Theorem \ref{bairdwood}. 

{\em Proof of Theorem $\ref{bairdwood}$:} Define
$$U=\{p\in\S^3:\rank(df_p)=2\}.$$
Note that $U$ is open and that $f : U \to f(U)\subset \mathbb{S}^2\subset\R^3$ is a weakly conformal submersion with totally geodesic fibers by Lemma \ref{lem_geofibers}. Let 
$\zeta$ be the unit vector field on $U$ defined as above.
We denote the fibers of $f$ passing through $p$ by
\[
\mathcal{A}_{p} \doteq \{ q \in \S^3 : f(q) = f(p) \}.
\]
Since $f$ is real analytic, each $\mathcal{A}_{p}$ is the zero set of the real analytic function $h: U\to\R$ given by
$$h(q)=\|f(q)-f(p)\|^2.$$
By \L ojasiewicz's structure theorem \cite[Theorem 6.3.3]{kr}, there exists a neighbourhood of $p$ where  
$\mathcal{A}_{p}$ decomposes as the disjoint union
\[
\mathcal{A}_{p} = \mathcal{V}_0 \cup \mathcal{V}_1 \cup \mathcal{V}_2 \cup \mathcal{V}_3,
\]
where $\mathcal{V}_d$, $0 \le d \le 3$, is either empty or a finite disjoint union of $d$-dimensional real analytic sub-varieties without boundary.
Moreover, up to shrinking the neighbourhood of $p$, if $\mathcal{V}_d \neq \emptyset$ then $\overline{\mathcal{V}}_d$ contains the union of all lower dimensional strata $\mathcal{V}_j$ for $0 \le j \le d-1$.
A point $q \in \mathcal{A}_{p}$ is called a \emph{regular point of dimension $d$} 
if there exists a neighborhood $\Omega$ of $q$ such that $\Omega \cap \mathcal{A}_{p}$ is a connected 
$d$-dimensional real analytic submanifold of $\Omega$. In particular, a regular point of dimension zero is an isolated point of $\mathcal{V}$. Otherwise, $q$ is called a 
\emph{singular point}. 

\begin{remark}\label{rem_singular}
%	The set of singular points is locally a finite union of submanifolds.
	 By the last property we quote of \L ojasiewicz's theorem, each singular point is a limit of regular points lying in the same connected component of $\mathcal{A}_{p}$. In particular, $\mathcal{A}_{p}$ must have at least one regular point.
\end{remark}
Since $|df|^2$ is a real analytic function as well, the above discussion also applies to
\[
\mathcal{A}_{df} \doteq \{ p \in \S^3 : |df_p| = 0 \}.
\]
Notice that
\[
df^*  df \sim 
\begin{pmatrix}
	\lambda^2 & 0 & 0 \\
	0 & \lambda^2 & 0 \\
	0 & 0 & 0
\end{pmatrix},
\]
therefore the rank of the differential can be either $0$ or $2$ and
$\mathcal{A}_{df}=\S^3\setminus U.$

Suppose henceforth that $f$ in non-constant.

\begin{lemma}\label{A_df in A_f}
	The map $f$ is constant on connected components of $\mathcal{A}_{df}$.
\end{lemma}
\begin{proof} The result is trivial for isolated points, i.e. regular points of dimension $0$.
	Let $p$ be a regular point of $\mathcal{A}_{df}$ of dimension $d \geq 1$. There exists an open connected neighborhood $\Omega$ of $p$ such that $\Omega \cap \mathcal{A}_{df}$ is a connected $d$-dimensional submanifold of $\Omega$. For any curve
	$\gamma : [0,1] \to \Omega \cap \mathcal{A}_{df}$
	we have that
	$(f \circ \gamma)' \equiv 0$
	everywhere, therefore $f \circ \gamma$ is constant. Since any two points of $\Omega \cap \mathcal{A}_{df}$ can be connected by a smooth curve, it must be that $f$ is constant on connected sets of regular points of $\mathcal{A}_{df}$. This can be extended to include singular points since $f$ is continuous.
\end{proof}

\begin{lemma}\label{no 3d points}
The fibers of $f$ have no regular points of dimension $2$ or  $3$.
\end{lemma}
\begin{proof}
Regular points of dimension $3$ would force $f$ to be constant by analyticity.
Let $p$ be a regular point of $\mathcal{A}_{p}$ of dimension $2$. Then there exists a connected open neighborhood $\Omega$ of $p$ such that $\mathcal{A}_{p} \cap \Omega$ is a $2$-dimensional submanifold of  $\Omega$. Furthermore, notice that $\rank df\le 1$ implies $\rank df=0$ on $\mathcal{A}_{p} \cap \Omega$.
It follows that $f$ is a solution of the Cauchy problem
\begin{equation*}\label{cauchy problem}
	\begin{cases}
		\tau(f) \!\!\!\!\!&=\,\,\,\, 0, \,\,\,\,\,\,\,\,\,\,\,\,\,\,\,\,\text{in}\,\,\,\,\,\, \Omega, \\
		df &= \,\,\,\,0, \,\,\,\,\,\,\,\,\,\,\,\,\,\,\,\,\text{in}\,\,\,\,\,\, \mathcal{A}_{p} \cap \Omega, \\
		f &= \,\,\,f(p), \,\,\,\,\,\,\,\,\,\text{in}\,\,\,\,\,\,\, \mathcal{A}_{p} \cap \Omega.
	\end{cases}
\end{equation*}
By the Cauchy-Kowalewsky theorem, this problem admits the unique solution 
$f = f(p)$ in $\Omega$. Therefore $f$ is again constant by analyticity,
which is a contradiction.
\end{proof}

\begin{lemma}\label{extendz}
The fibers of $f$ consist of non-intersecting great circles in $\mathbb{S}^3$, and the vector field $\zeta$ admits a smooth extension on the entire $\S^3$.
\end{lemma}

\begin{proof}
Concerning the extension, since $\zeta$ is smooth on $U$ we only have to consider points
$p\in\mathcal{A}_{df}=\S^3\setminus U$. By Lemma \ref{A_df in A_f} the connected component
of $\mathcal{A}_{df}$ containing $p$ lies in $\mathcal{A}_p$.

{\bf Claim 1:} {\em $\mathcal{A}_{p}$ contains at least one great circle passing through $p$. In particular, $p$ cannot be a regular point of dimension $0$ for $\mathcal{A}_{p}$.} 

{\em Proof of Claim 1.} Since $U$ is dense in $\S^3$, we can pick a sequence of points $(p_n)_{n\in\mathbb{N}} \subset U$ converging to $p$. By Lemma \ref{lem_geofibers}, for each $p_n$ there is a great circle $\gamma_n \subset \mathcal{A}_{p_n}$ passing through it, say parametrized by arclength and with $\gamma_n(0) = p_n$. Up to a subsequence, by the Ascoli-Arzel\'a theorem $\gamma_n$ converges uniformly to a great circle $\gamma_\infty$ passing through $p$, and by continuity $f(q) = f(p)$ for each $q \in \gamma_\infty$. \hfill$\circledast$

{\bf Corollary to Claim 1:} {\em If $p$ is a regular point of dimension $1$ for $\mathcal{A}_p$, then any other sub-sequential limit $\overline{\gamma}_{\infty}: \S^{1} \to \S^{3}$ of $\gamma_n$} must satisfy
\[ \overline{\gamma}_{\infty} = \pm \gamma_{\infty}. \]

{\bf Claim 2:} {\em If $p$ is a regular point of dimension $1$ for $\mathcal{A}_p$, there exists an open geodesic disc $D$ centered at $p$, transversal to the fibers of $f$, whose unit normal $\nu$ can be chosen such that $\nu|_{p} = \gamma'_{\infty}(0)$. }

{\em Proof of Claim 2.} Let
$$\mathcal{H}_{p} = (T_{p}\gamma_{\infty}(\S^{1}))^{\perp} \subset T_{p}\S^{3}.$$
Denote by $B_0^{\mathcal{H}}(r)$ the ball of radius $r$ in $\mathcal{H}_p$ centered at the origin. Suppose that none of the discs
\[ D_{n} \doteq \exp_{p}(B_0^{\mathcal{H}}(\tfrac{1}{n})) \]
are transversal to the fibers of $f$. This implies that there exists a sequence of points $(p_n)_{n\in\mathbb{N}}$ such that
\[ p_{n} \tends p \quad\text{and}\quad \zeta|_{p_{n}} \in T_{p_{n}}D_{n}, \quad\text{for all $n\in\N$}. \]
Following the same reasoning as in the proof of Claim 1, we can construct a geodesic that is perpendicular to $\gamma_{\infty}$ at $p$, and on which $f$ is constant. This is a contradiction since $p$ is a regular point of $\mathcal{A}_{p}$. \hfill$\circledast$

{\bf Claim 3:} {\em If $p$ is a regular point of dimension $1$ for $\mathcal{A}_p$, then the entire sequence $(\gamma_n)_{n\in\mathbb{N}}$ described in Claim 1 converges to the great circle $\gamma_{\infty}$.}

{\em Proof of Claim 3.} Let $D$ be the disc given in Claim 2. We shrink $D$ if necessary, so that
\[ \langle \zeta, \nu \rangle \geq \delta > 0, \quad\text{on $D \setminus \{ p \}$}. \]
The fact that this quantity does not approach zero near $p$ is guaranteed by the proof of Claim 2. We now extend $\nu$ by parallel transport along the geodesic fibers of $f$, so that
\[ \langle \zeta, \nu \rangle \geq \delta, \quad\text{on $\Omega \setminus \mathcal{A}_{p}$}. \]
If $\overline{\gamma}_{\infty} : \S^{1} \to \S^{3}$ is the limit geodesic of another subsequence
$(\gamma_{n_l})_{l\in\mathbb{N}}$, this implies that
\[ \langle \overline{\gamma}'_{\infty}(0), \gamma'_{\infty}(0) \rangle = \lim_{l \to \infty} \langle \zeta, \nu \rangle(p_{n_{l}}) \geq \delta > 0.  \]
It follows from the Corollary to Claim 1 that
$\overline{\gamma}_\infty \equiv \gamma_\infty,$
as claimed.

{\bf Claim 4:} {\em $\zeta$ admits a smooth extension to any $p \in \mathcal{A}_{df}$ which is a regular point of its fiber.}

{\em Proof of Claim 4.} By Lemma \ref{no 3d points} and Claim 1, any such $p$ must be a regular point of dimension $1$ for $\mathcal{A}_p$, hence by Claim 3 the vector field $\zeta$ admits a continuous extension to all the points in a small enough neighbourhood $\Omega$ of $p$ (one for which $\mathcal{A}_p \cap \Omega$ is a single curve). Since the vector field $\zeta$ is harmonic in $U$, according to results of Meier \cite{meier1,meier2} the extension is smooth in $\Omega$.
\hfill$\circledast$

{\bf Claim 5:} {\em For each $p\in\S^3$, the set $\mathcal{A}_p$ is a disjoint union of great circles. In
particular, $\mathcal{A}_p$ does not have singular points.}

{\em Proof of Claim 5.} By Claims 1 and 3, we only  need to consider the case where $p$ is a singular point of $\mathcal{A}_p$.
By {\L}ojasiewicz's Theorem and Lemma \ref{no 3d points}, the point $p$ is in the closure of a union of disjoint curves. Each of them is therefore made up of regular points of $\mathcal{A}_p$ of dimension $1$ and thus, by Claim 1, it is a piece of great circle contained in $\mathcal{A}_p$. To prove the claim it is therefore enough to exclude the possibility that
$\mathcal{A}_p$ contains two great circles $\sigma_j :  \mathcal{S}^1\to\S^3$, $j\in\{1,2\}$, meeting transversely at $p$ at $t=0$, i.e.
$$
\sigma_1(0)=p=\sigma_2(0)\quad\text{and}\quad \langle \sigma'_1(0),\sigma'_2(0)\rangle\neq{\pm 1}.
$$
Note that
$(f \circ \sigma_1) \equiv (f \circ \sigma_2) \equiv f(p).$
Let $q$ be a regular point of $\mathcal{A}_{p}$ in the image of $\sigma_1$. We have shown in Claim 4 that $\zeta$ admits a smooth extension in a neighborhood $\Omega$ of $q$. Let $D$ be a disk centered at $q$ as described in Claim 2, and proceed as in the proof of Claim 3 so that
\[ \langle \zeta, \nu \rangle \geq \delta > 0, \quad\text{on $\Omega$}. \]
In particular,  $\zeta|_{D}$ is smooth, and $\zeta|_{\Omega}$ is obtained by parallelly transporting $\zeta|_{D}$ along the geodesic fibers of the map $f$ in $\Omega$. For $\epsilon$ small enough, we can thus extend $\zeta$ in a tubular neighborhood of $\sigma_1(-\epsilon, \epsilon)$ foliated by the geodesic fibers $\exp_x (t\zeta|_x)$, $x\in D$. In particular, this gives a smooth extension of $\zeta$ to $\sigma_2([0,\tilde\epsilon))$, for $\tilde\epsilon$ small enough. Consider the map
$F : [0,\tilde\epsilon) \times (-\epsilon,\epsilon) \to \S^{3}$
given by
\[ F(s,t) \doteq \exp_{\exp_{p}(s \sigma'_2(0))}(t\zeta). \]
Shrinking $\epsilon$ and $\tilde\epsilon$ further if necessary, the image of $F$ is an embedded $2$-dimensional surface in $\S^{3}$, and the maps $F(s_0, \,\cdot\,) : (-\epsilon,\epsilon) \to \S^{3}$ are geodesic fibers of $f$ (they are integral curves of $\zeta$ away from the image of $\sigma_2$). Since $f$ is also constant on
$F(\,\cdot\,,0) = \sigma_2|_{[0,\tilde\epsilon)},$
it must be constant on a $2$-dimensional surface in $\S^{3}$, which contradicts Lemma \ref{no 3d points}. \hfill$\circledast$

By Claim 5 the extension of $\zeta$ provided in Claim 4 is in fact an extension to the entire $\S^3$.
This concludes the proof of the lemma.
\end{proof}

Let us consider the quotient map $\pi : \S^3 \to \S^2$ generated by $\zeta$, which by Lemma \ref{lem_topological} is a smooth submersion. We claim that there exists a smooth map 
$g : \S^2 \to \S^2$ such that
\begin{equation} \label{mapg}
f = g \circ \pi.
\end{equation}
The map $g$ is constructed as follows. Given $x \in \S^2$, consider the 
great circle
$c_x \doteq \pi^{-1}(x).$ 
Since a fiber of $\pi$ is a connected component of a fiber of $f$, the circle $c_x$ is mapped by $f$ to a point 
$\tilde{x} = f(c_x)$. We then define
\[
g(x) \doteq \tilde{x}.
\]
Since $\pi$ is a submersion, locally around $x$ it admits a smooth section
$\sigma : V \subset \S^2 \to \S^3$, and thus $g \equiv f \circ \sigma$ is smooth. Moreover, $g$ 
has rank $2$ at every point of the open set $\pi(U)$, in particular, $g$ is a global diffeomorphism 
if $f$ is a submersion. From Theorem \ref{fms} it follows that $\pi$ is a Hopf fibration, and it is 
immediate that $g$ must be conformal. 

\section{Proofs of the main theorems}

\subsection{Proof of Theorem \ref{thmA}}
Observe that the function
$u^2-4v^2=(\lambda^2-\mu^2)^2$
is real analytic and vanishes precisely at points where $\lambda=\mu$.
Hence, if $\lambda=\mu$ in an open neighborhood, then $\lambda=\mu$ everywhere on $\S^3$.
Assume now to the contrary that the continuous function
$$w = (\lambda - \mu)^2 \geq 0$$
is not identically zero on $\S^{3}$. Notice that if 
$v^2=\lambda^2\mu^2 = 0$
on an open subset of $\S^{3}$, then it is zero everywhere by analyticity and from Lemma \ref{lapu} we deduce that $f$ is constant. Therefore we may assume that the open set
$$
U \doteq \{ x \in \S^{3} : \lambda(x) >\mu(x)> 0 \}
$$
is non-empty, and observe that $w$ is smooth on $U$. Assume that $w$ attains its maximum at a point $p_0$. If $p_0 \in U$, then by Lemma \ref{lapw} we have that
\begin{equation}\label{A1}
  b_{1k}^{4}(p_0) = b_{2k}^{5}(p_0),\quad\text{for}\,\,\, k\in\{1,2,3\}.
\end{equation}
Moreover, the Laplacian of $w$ on $U$ is given by
\begin{equation}\label{lapw eq}
  \begin{aligned}[b]
    \frac{1}{2}\Delta w &= w(\lambda\mu+2) + |B|^{2} - 2\sum_{k=1}^{3} (b_{1k}^{4}b_{2k}^{5} - b_{1k}^{5}b_{2k}^{4}) - \frac{\mu}{\lambda}\sum_{k=1}^{3} (b_{3k}^{4})^2 - \frac{\lambda}{\mu} \sum_{k=1}^{3} (b_{3k}^{5})^2. \\
    &= w(\lambda\mu+2) + (b_{11}^{4} - b_{12}^{5})^2 + (b_{12}^{4} - b_{22}^{5})^2 + (b_{11}^{5} + b_{12}^{4})^2 + (b_{12}^{5} + b_{22}^{4}) ^2 \\
    &\quad + \Big( 2 - \frac{\mu}{\lambda} \Big)\Big\{(b_{13}^{4})^2 + (b_{23}^{4})^2\Big\} + \Big( 2 - \frac{\lambda}{\mu} \Big)\Big\{(b_{13}^{5})^2 + (b_{23}^{5})^2\Big\} \\
    &\quad - 2b_{13}^{4}b_{23}^{5} + 2b_{13}^{5}b_{23}^{4} + \Big( 1 - \frac{\mu}{\lambda} \Big)(b_{33}^{4})^2 + \Big( 1 - \frac{\lambda}{\mu} \Big)(b_{33}^{5})^2.
  \end{aligned}
\end{equation}
By \eqref{A1}, harmonicity and Young's inequality we obtain that at $p_0\in U$,
\begin{equation}\label{lapww eq}
\frac{1}{2}\Delta w \geq w(\lambda\mu+2) + \Big( 2 - \frac{\mu}{\lambda} - \frac{\lambda}{\mu} \Big) \Big(|B_{13}|^2 + |B_{23}|^2 + |B_{33}|^2 \Big) \geq
\Big( \lambda\mu + 2 - \frac{|B|^2}{\lambda\mu} \Big)w>0,
\end{equation}
which is a contradiction. Since we have assumed that $w$ is not identically zero, we have
\begin{equation}\label{ccc}
\lambda(p_0) > \mu(p_0) = 0.
\end{equation}
(1) Let us examine the case
\begin{equation}\label{a}
|B|^2\le\lambda\mu(\lambda\mu+a).
\end{equation}
In particular this implies that $|B|$ tends to zero as $\mu$ approaches zero. The set
$$V \doteq \{ \mu=0 \} = \{ v^2 = 0 \} \subset \S^{3}$$
is an analytic set of measure zero, and by our pinching it follows that
\[ |\nabla w|^2 \leq C(\lambda - \mu)^2|B|^2 \leq C(\lambda - \mu)^2\lambda\mu(\lambda\mu+a),  \]
whereby $\nabla w$ can be continuously extended by zero on $V$. It follows that $w$ is locally constant on $V$. By Remark \ref{rem_singular}, up to slightly moving $p_0$ we can assume that $p_0$ is a regular point of $V$. In particular, regardless of the dimension of $p_0$ as a regular point, there exists a small ball $D \subset U$ with $p_0 \in \partial D$. Now from \eqref{lapww eq} and \eqref{a}, we get that
$$
\begin{array}{lcl}
\frac{1}{2}\Delta w & \geq & w(\lambda\mu + 2) + 2|B|^2 - (\lambda^2 + \mu^2)(\lambda\mu + \alpha) \\[0.3cm]
& \tends & (2-\alpha)\lambda^2(p_0) > 0 \qquad \text{as } \, p \to p_0.
\end{array}
$$
For this reason, up to shrinking $D$ we can assume that 
$$
\Delta w>0, \quad\text{on}\,\,\, D.
$$
If $\xi$ is the unit normal to the boundary of $D$ at $p_0$, then by Hopf's Boundary Point Lemma, see for example \cite{protter}, we get that
\[ 0 = \langle \nabla w, \xi \rangle = \frac{\partial w}{\partial \xi} > 0, \]
which is a contradiction unless $w$ is a positive constant on $D$. However, in this latter case, from \eqref{lapww eq} we arrive again at contradiction. Hence $w=0$ and the result follows from Theorem \ref{bairdwood}.

(2) This pinching condition implies that $\lambda\mu$ is 
positive everywhere on $\S^{3}$, therefore $\mu$ cannot be zero which contradicts \eqref{ccc}.
 Thus $w=0$ and the result follows from Theorem \ref{bairdwood}.

\subsection{Proof of Theorem \ref{thmC}}
We proceed exactly as in the proof of Theorem \ref{thmA}. Note that, under our assumptions,
$B_{33}=0$ on $U$.
Applying Young's inequality to equation \eqref{lapw eq} we get
\begin{equation}\label{lapw ineq}
\frac{1}{2}\Delta w \geq w(\lambda\mu+2) + \Big( 2 - \frac{\mu}{\lambda} - \frac{\lambda}{\mu} \Big) \Big(|B_{13}|^2 + |B_{23}|^2 \Big) \geq
\Big( \lambda\mu + 2 - \frac{|B|^2}{2\lambda\mu} \Big)w \geq 0,
\end{equation}
this time on the entirety of $U$. Using Hopf's Boundary Point Lemma, as in Theorem \ref{thmA}, we get that $w$ must be constant on $\S^{3}$. Now by \eqref{lapw ineq} we have that
\[ \left( \lambda\mu + 2 - \frac{|B|^2}{2 \lambda\mu} \right)w \equiv 0, \]
which, since we have assumed that $w$ is not identically zero, implies that
\[ |B|^2 \equiv 2 \lambda\mu( \lambda\mu+2). \]
Then by Lemma \ref{lapu} we get that
\[ \frac{1}{2}\Delta u = 2(\lambda^2+\mu^2) + 4\lambda\mu = 2(\lambda + \mu)^2 \geq 0, \]
and by the maximum principle $f$ must be constant, which is again a contradiction. Thus
$w=0$ and the result follows from Theorem \ref{bairdwood}.

\subsection{Proof of Theorem \ref{thmB}}
Denote as usual by $\lambda \ge \mu$ the two largest singular values. Assume that the open set $U = \{\lambda > \mu\}$ is non-empty. By Lemma \ref{7.2}, $\lambda^2$ and $\mu^2$ are smooth on $U$. Moreover, by assumption  
\begin{equation}\label{wein}
W(\lambda,\mu)=0, \qquad \mu W_x(\lambda,\mu) +\lambda W_y(\lambda,\mu) \neq 0 \quad\text{on $U$}.
\end{equation}
\textbf{Claim 1:} {\em The singular value $\mu$ is $C^{1}$ on $U$. }

\textit{Proof of Claim 1.} Differentiability may fail on the closed analytic subset 
\[ 
\Sigma = \{ \lambda > \mu = 0 \} \subset U, 
\]
(the zero set of the analytic function $v^2$ in \eqref{vdef}, see Lemma \ref{lapu}), on which in particular $\lambda$ is smooth; see \cite{nomizu}. Differentiating \eqref{wein} away from $\Sigma$ we get that
\[ 
W_x(\lambda, \mu)\,\nabla\lambda + W_y(\lambda, \mu)\nabla\mu = 0. 
\]
Since $W$ is $C^{1}$ and $\lambda$ and $\mu$ are continuous, the composition $W_y(\lambda,\mu)$ is continuous up to $\Sigma$, and the constraint in \eqref{wein} forces it to be non-zero in the entire $U$. This implies the identity
\[ 
\nabla\mu = -\frac{W_x(\lambda, \mu)}{W_y(\lambda, \mu)}\,\nabla\lambda \qquad \text{on } \, U \backslash \Sigma.
\]
Since the right-hand side is well defined and continuous on $U$, it follows that $\mu$ admits a
$C^1$ extension to $U$. This proves our claim. \hfill$\circledast$

\textbf{Claim 2:} {\em The continuous function $\varrho: \S^3\to\R$ given by
\[ \varrho=\lambda^2-\mu^2\ge0 \]
is identically zero on $\S^{3}$. }

\textit{Proof of Claim 2.} Suppose to the contrary that this is not the case. Then the open set $U$ from Claim 1 is non-empty, and by Lemma \ref{7.2} the function $\rho$ is smooth and attains its maximum at a point $p_0 \in U$. Differentiating $\rho$ and using the relation \eqref{wein} at $p_0$ we get that
\[
\begin{cases}\label{wein1}
   W_x\,\alpha_{k}(\lambda) + W_y\,\alpha_{k}(\mu) = 0, \\
  \lambda\alpha_{k}(\lambda) - \mu\alpha_{k}(\mu) = 0,
\end{cases}
\]
whereby, by the constraint in \eqref{wein}, we obtain that at $p_0$
\begin{equation}\label{zero derivatives}
  \alpha_{k}(\lambda) = \alpha_{k}(\mu) = 0.
\end{equation}
Equivalently, by Lemma \ref{7.1}, at $p_0$ we have that
\[ b_{11}^{4} = b_{12}^{4} = b_{13}^{4} = b_{12}^{5} = b_{22}^{5} = b_{23}^{5} = 0, \]
which in particular also implies that
\[ b_{11}^{5} = -b_{33}^{5} \quad\text{and}\quad b_{22}^{4} = -b_{33}^{4}. \]
Substituting into the equation from Lemma \ref{lapbarw} we immediately obtain that
\begin{equation*}
\Delta\rho(p_0) \geq 4\rho(p_0) > 0,
\end{equation*}
which is a contradiction. \hfill$\circledast$

By Claim 2, $f$ is a weakly conformal harmonic map. Hence, by Theorem \ref{bairdwood}, the map $f$ is the composition of the Hopf fibration with a conformal map of $\S^2$.

\subsection{Proof of Theorem \ref{thmD}}
First assume that the $2$-dilation of $f$ is a constant $c$. Then $f$ satisfies the assumptions in Theorem \ref{thmB} with 
	\[
	W(x,y) = xy - c. 
	\]
	Hence $f$ is weakly conformal, i.e. $\lambda \equiv \mu$ everywhere on $\S^3$. From $W(\lambda,\mu) = 0$ we therefore deduce that $\lambda^2$ is constant. Escobales' result \cite{escobales} guarantees that $f$ is a Hopf fibration.
	
	Next assume that the energy density $u$ is a positive constant $c$. From Lemma \ref{lapu} we get 
	\begin{equation}\label{eq_B_tot}
	|B|^2 = 2(v^2-u). 
	\end{equation}
	Observe that $\mu$ cannot vanish at any point $p$, otherwise $v$ would vanish at $p$ and thus from \eqref{eq_B_tot} we get $u= |B|^2 =0$ at $p$, contradiction. Consequently, $f$ is a submersion. Setting
	$$W(x,y) = x^2 + y^2 - c,$$ even though \eqref{eq_Weingarten} is not satisfied on $\{x>y=0\}$, nevertheless 
\[
\mu W_x(\lambda,\mu) + \lambda W_y(\lambda,\mu) \neq 0 \qquad \text{holds everywhere on } \, \S^3.
\]
This suffices to follow the proof of Theorem \ref{thmB} and conclude that $f$ is weakly conformal. From $\lambda \equiv \mu$ and $u=\lambda^2+\mu^2$ constant we deduce that $\lambda$ and $\mu$ are constant. The conclusion follows by Escobales' result.

% Literaturliste
%%%%%%%%%%%%%%%%%%%%%%%%%%%%%%%%%%%%%%%%%%%%%%%%%%%%%%%%%%%%%%%%%%%%%%%%%

\end{document}